\documentclass[12pt,reqno]{amsart}
\usepackage{amssymb,amscd,url}
\usepackage{xcolor}    
\usepackage[all]{xy}  
\usepackage{bbm}      
\usepackage{fancybox} 
\usepackage{upref}     

\counterwithin{figure}{section}  

\begin{document}

\allowdisplaybreaks


\title[Automorphisms of Commuting Polynomial Maps]
      {Automorphism Groups of Commuting Polynomial Maps of the Affine Plane}
\date{\today}
\author[J.H. Silverman]{Joseph H. Silverman}
\email{joseph\_silverman@brown.edu}
\address{Mathematics Department, Box 1917
  Brown University, Providence, RI 02912 USA.
  ORCID: https://orcid.org/0000-0003-3887-3248}



\hyphenation{ca-non-i-cal semi-abel-ian}


\newtheorem{theorem}{Theorem}[section]
\newtheorem{lemma}[theorem]{Lemma}
\newtheorem{sublemma}[theorem]{Sublemma}
\newtheorem{conjecture}[theorem]{Conjecture}
\newtheorem{proposition}[theorem]{Proposition}
\newtheorem{corollary}[theorem]{Corollary}
\newtheorem{assumption}[theorem]{Assumption}
\newtheorem*{claim}{Claim}

\theoremstyle{definition}
\newtheorem{definition}[theorem]{Definition}
\newtheorem*{intuition}{Intuition}
\newtheorem{example}[theorem]{Example}
\newtheorem{remark}[theorem]{Remark}
\newtheorem{question}[theorem]{Question}

\theoremstyle{remark}
\newtheorem*{acknowledgement}{Acknowledgements}


\newenvironment{notation}[0]{%
  \begin{list}%
    {}%
    {\setlength{\itemindent}{0pt}
     \setlength{\labelwidth}{4\parindent}
     \setlength{\labelsep}{\parindent}
     \setlength{\leftmargin}{5\parindent}
     \setlength{\itemsep}{0pt}
     }%
   }%
  {\end{list}}

\newenvironment{parts}[0]{%
  \begin{list}{}%
    {\setlength{\itemindent}{0pt}
     \setlength{\labelwidth}{1.5\parindent}
     \setlength{\labelsep}{.5\parindent}
     \setlength{\leftmargin}{2\parindent}
     \setlength{\itemsep}{0pt}
     }%
   }%
  {\end{list}}
\newcommand{\Part}[1]{\item[\upshape#1]}

\def\Case#1#2{%
\paragraph{\textbf{\boldmath Case #1: #2.}}\hfil\break\ignorespaces}

\renewcommand{\a}{\alpha}
\renewcommand{\b}{\beta}
\newcommand{\g}{\gamma}
\renewcommand{\d}{\delta}
\newcommand{\e}{\epsilon}
\newcommand{\f}{\varphi}
\newcommand{\fhat}{\hat\varphi}
\newcommand{\bfphi}{{\boldsymbol{\f}}}
\renewcommand{\l}{\lambda}
\renewcommand{\k}{\kappa}
\newcommand{\bfkappa}{{\boldsymbol{\kappa}}}
\newcommand{\lhat}{\hat\lambda}
\newcommand{\m}{\mu}
\newcommand{\bfmu}{{\boldsymbol{\mu}}}
\renewcommand{\o}{\omega}
\newcommand{\bfo}{{\boldsymbol{\omega}}}
\newcommand{\bfpi}{{\boldsymbol{\pi}}}
\renewcommand{\r}{\rho}
\newcommand{\bfrho}{{\boldsymbol{\rho}}}
\newcommand{\rbar}{{\bar\rho}}
\newcommand{\s}{\sigma}
\newcommand{\sbar}{{\bar\sigma}}
\newcommand{\bfsigma}{{\boldsymbol{\sigma}}}
\renewcommand{\t}{\tau}
\newcommand{\z}{\zeta}
\newcommand{\bfzeta}{{\boldsymbol\zeta}}

\newcommand{\D}{\Delta}
\newcommand{\G}{\Gamma}
\newcommand{\F}{\Phi}
\renewcommand{\L}{\Lambda}

\newcommand{\ga}{{\mathfrak{a}}}
\newcommand{\gb}{{\mathfrak{b}}}
\newcommand{\frakF}{{\mathfrak{F}}}
\newcommand{\frakg}{{\mathfrak{g}}}
\newcommand{\frakh}{{\mathfrak{h}}}
\newcommand{\gM}{{\mathfrak{M}}}
\newcommand{\gn}{{\mathfrak{n}}}
\newcommand{\gp}{{\mathfrak{p}}}
\newcommand{\gP}{{\mathfrak{P}}}
\newcommand{\gq}{{\mathfrak{q}}}
\newcommand{\frakt}{{\mathfrak{t}}}

\newcommand{\Abar}{{\bar A}}
\newcommand{\Bbar}{{\bar B}}
\newcommand{\Cbar}{{\bar C}}
\newcommand{\Ebar}{{\bar E}}
\newcommand{\Gbar}{{\bar G}}
\newcommand{\kbar}{{\bar k}}
\newcommand{\Kbar}{{\bar K}}
\newcommand{\Pbar}{{\bar P}}
\newcommand{\Sbar}{{\bar S}}
\newcommand{\Tbar}{{\bar T}}
\newcommand{\gbar}{{\bar\gamma}}
\newcommand{\lbar}{{\bar\lambda}}
\newcommand{\ybar}{{\bar y}}
\newcommand{\phibar}{{\bar\f}}
\newcommand{\nubar}{{\overline\nu}}

\newcommand{\Acal}{{\mathcal A}}
\newcommand{\Bcal}{{\mathcal B}}
\newcommand{\Ccal}{{\mathcal C}}
\newcommand{\Dcal}{{\mathcal D}}
\newcommand{\Ecal}{{\mathcal E}}
\newcommand{\Fcal}{{\mathcal F}}
\newcommand{\Gcal}{{\mathcal G}}
\newcommand{\Hcal}{{\mathcal H}}
\newcommand{\Ical}{{\mathcal I}}
\newcommand{\Jcal}{{\mathcal J}}
\newcommand{\Kcal}{{\mathcal K}}
\newcommand{\Lcal}{{\mathcal L}}
\newcommand{\Mcal}{{\mathcal M}}
\newcommand{\Ncal}{{\mathcal N}}
\newcommand{\Ocal}{{\mathcal O}}
\newcommand{\Pcal}{{\mathcal P}}
\newcommand{\Qcal}{{\mathcal Q}}
\newcommand{\Rcal}{{\mathcal R}}
\newcommand{\Scal}{{\mathcal S}}
\newcommand{\Tcal}{{\mathcal T}}
\newcommand{\Ucal}{{\mathcal U}}
\newcommand{\Vcal}{{\mathcal V}}
\newcommand{\Wcal}{{\mathcal W}}
\newcommand{\Xcal}{{\mathcal X}}
\newcommand{\Ycal}{{\mathcal Y}}
\newcommand{\Zcal}{{\mathcal Z}}

\renewcommand{\AA}{\mathbb{A}}
\newcommand{\BB}{\mathbb{B}}
\newcommand{\CC}{\mathbb{C}}
\newcommand{\FF}{\mathbb{F}}
\newcommand{\GG}{\mathbb{G}}
\newcommand{\NN}{\mathbb{N}}
\newcommand{\PP}{\mathbb{P}}
\newcommand{\QQ}{\mathbb{Q}}
\newcommand{\RR}{\mathbb{R}}
\newcommand{\ZZ}{\mathbb{Z}}

\newcommand{\bfa}{{\boldsymbol a}}
\newcommand{\bfb}{{\boldsymbol b}}
\newcommand{\bfc}{{\boldsymbol c}}
\newcommand{\bfd}{{\boldsymbol d}}
\newcommand{\bfe}{{\boldsymbol e}}
\newcommand{\bff}{{\boldsymbol f}}
\newcommand{\bfg}{{\boldsymbol g}}
\newcommand{\bfi}{{\boldsymbol i}}
\newcommand{\bfj}{{\boldsymbol j}}
\newcommand{\bfk}{{\boldsymbol k}}
\newcommand{\bfm}{{\boldsymbol m}}
\newcommand{\bfn}{{\boldsymbol n}}
\newcommand{\bfp}{{\boldsymbol p}}
\newcommand{\bfq}{{\boldsymbol q}}
\newcommand{\bfr}{{\boldsymbol r}}
\newcommand{\bfs}{{\boldsymbol s}}
\newcommand{\bft}{{\boldsymbol t}}
\newcommand{\bfu}{{\boldsymbol u}}
\newcommand{\bfv}{{\boldsymbol v}}
\newcommand{\bfw}{{\boldsymbol w}}
\newcommand{\bfx}{{\boldsymbol x}}
\newcommand{\bfy}{{\boldsymbol y}}
\newcommand{\bfz}{{\boldsymbol z}}
\newcommand{\bfA}{{\boldsymbol A}}
\newcommand{\bfF}{{\boldsymbol F}}
\newcommand{\bfB}{{\boldsymbol B}}
\newcommand{\bfD}{{\boldsymbol D}}
\newcommand{\bfG}{{\boldsymbol G}}
\newcommand{\bfI}{{\boldsymbol I}}
\newcommand{\bfM}{{\boldsymbol M}}
\newcommand{\bfP}{{\boldsymbol P}}
\newcommand{\bfQ}{{\boldsymbol Q}}
\newcommand{\bfS}{{\boldsymbol S}}
\newcommand{\bfT}{{\boldsymbol T}}
\newcommand{\bfU}{{\boldsymbol U}}
\newcommand{\bfX}{{\boldsymbol X}}
\newcommand{\bfY}{{\boldsymbol Y}}
\newcommand{\bfzero}{{\boldsymbol{0}}}
\newcommand{\bfone}{{\boldsymbol{1}}}

\newcommand{\ad}{\operatorname{ad}}
\newcommand{\AffineAut}{\operatorname{\textup{\textsf{AffAut}}}}
\newcommand{\Aut}{\operatorname{Aut}}
\newcommand{\Berk}{{\textup{Berk}}}
\newcommand{\Birat}{\operatorname{Birat}}
\newcommand{\characteristic}{\operatorname{char}}
\newcommand{\codim}{\operatorname{codim}}
\newcommand{\Coef}{\operatorname{\textup{\textsf{Coef}}}}
\newcommand{\Commute}{\operatorname{\textup{\textsf{Com}}}}
\newcommand{\Crit}{\operatorname{Crit}}
\newcommand{\crit}{{\textup{crit}}}
\newcommand{\Deg}{\operatorname{\textup{\textsf{Deg}}}}
\newcommand{\diag}{\operatorname{diag}}
\newcommand{\Disc}{\operatorname{Disc}}
\newcommand{\Div}{\operatorname{Div}}
\newcommand{\Dom}{\operatorname{Dom}}
\newcommand{\dyn}{{\textup{dyn}}}
\newcommand{\End}{\operatorname{End}}
\newcommand{\Fbar}{{\bar{F}}}
\newcommand{\Fix}{\operatorname{Fix}}
\newcommand{\Fiber}{\operatorname{Fiber}}
\newcommand{\Fold}{{\textup{\textsf{F}}}}
\newcommand{\FOD}{\operatorname{FOD}}
\newcommand{\FOM}{\operatorname{FOM}}
\newcommand{\Gal}{\operatorname{Gal}}
\newcommand{\GITQuot}{/\!/}
\newcommand{\GL}{\operatorname{GL}}
\newcommand{\Hom}{\operatorname{Hom}}
\newcommand{\Index}{\operatorname{Index}}
\newcommand{\Imag}{\operatorname{\textup{\textsf{Im}}}}
\newcommand{\Image}{\operatorname{Image}}
\newcommand{\Isom}{\operatorname{Isom}}
\newcommand{\hhat}{{\hat h}}
\newcommand{\Ker}{{\operatorname{ker}}}
\newcommand{\Length}{\operatorname{Length}}
\newcommand{\Lie}[1]{{\mathcal{#1}}}
\newcommand{\Lift}{\operatorname{Lift}}
\newcommand{\Mat}{\operatorname{Mat}}
\newcommand{\maxplus}{\operatornamewithlimits{\textup{max}^{\scriptscriptstyle+}}}
\newcommand{\MOD}[1]{~(\textup{mod}~#1)}
\newcommand{\Mor}{\operatorname{Mor}}
\newcommand{\Moduli}{\mathcal{M}}
\newcommand{\MODULI}{\overline{\mathcal{M}}}
\newcommand{\Norm}{{\operatorname{\mathsf{N}}}}
\newcommand{\notdivide}{\nmid}
\newcommand{\normalsubgroup}{\triangleleft}
\newcommand{\NS}{\operatorname{NS}}
\newcommand{\onto}{\twoheadrightarrow}
\newcommand{\ord}{\operatorname{ord}}
\newcommand{\Orbit}{\mathcal{O}}
\newcommand{\Per}{\operatorname{Per}}
\newcommand{\Perp}{\operatorname{Perp}}
\newcommand{\PrePer}{\operatorname{PrePer}}
\newcommand{\PGL}{\operatorname{PGL}}
\newcommand{\Pic}{\operatorname{Pic}}
\newcommand{\prim}{\textup{prim}}
\newcommand{\Prob}{\operatorname{Prob}}
\newcommand{\Proj}{\operatorname{Proj}}
\newcommand{\Qbar}{{\bar{\QQ}}}
\newcommand{\rank}{\operatorname{rank}}
\newcommand{\Rat}{\operatorname{Rat}}
\newcommand{\Real}{\operatorname{\textup{\textsf{Re}}}}
\newcommand{\Resultant}{\operatorname{Res}}
\newcommand{\Residue}{\operatorname{Residue}} 
\renewcommand{\setminus}{\smallsetminus}
\newcommand{\sgn}{\operatorname{sgn}}
\newcommand{\SL}{\operatorname{SL}}
\newcommand{\Span}{\operatorname{Span}}
\newcommand{\Spec}{\operatorname{Spec}}
\renewcommand{\ss}{{\textup{ss}}}
\newcommand{\stab}{{\textup{stab}}}
\newcommand{\Stab}{\operatorname{Stab}}
\newcommand{\Support}{\operatorname{Supp}}
\newcommand{\Sym}{\operatorname{Sym}}  
\newcommand{\Top}{\operatorname{\textup{\textsf{Top}}}}
\newcommand{\tors}{{\textup{tors}}}
\newcommand{\Trace}{\operatorname{Trace}}
\newcommand{\trianglebin}{\mathbin{\triangle}} 
\newcommand{\tr}{{\textup{tr}}} 
\newcommand{\UHP}{{\mathfrak{h}}}    
\newcommand{\val}{\operatorname{val}} 
\newcommand{\wt}{\operatorname{wt}} 
\newcommand{\zbar}{\overline{z}}
\newcommand{\barz}{\overline{z}}
\newcommand{\<}{\langle}
\renewcommand{\>}{\rangle}

\newcommand{\pmodintext}[1]{~\textup{(mod}~#1\textup{)}}
\newcommand{\ds}{\displaystyle}
\newcommand{\longhookrightarrow}{\lhook\joinrel\longrightarrow}
\newcommand{\longonto}{\relbar\joinrel\twoheadrightarrow}
\newcommand{\SmallMatrix}[1]{%
  \left(\begin{smallmatrix} #1 \end{smallmatrix}\right)}

\newcommand{\TABT}[1]{\begin{tabular}[t]{@{}l@{}}#1\end{tabular}}
\newcommand{\TAB}[1]{\begin{tabular}{@{}l@{}}#1\end{tabular}}
\newcommand{\TABC}[1]{\begin{tabular}{@{}c@{}}#1\end{tabular}}
\newcommand{\ARRAY}[1]{\begin{array}{@{}l@{}}#1\end{array}}

\def\COMMENT#1{}  

\definecolor{hdegcolor}{RGB}{255, 0, 255}    
\definecolor{hdegcolor}{RGB}{1, 50, 32}    
\newcommand{\hdegup}[2]{\overbrace{{\color{black}#2}\strut}^{\color{hdegcolor}\strut\ovalbox{\textup{{$\scriptscriptstyle#1$}}}}}
\newcommand{\hdegdown}[2]{\underbrace{{\color{black}#2}\strut}_{\color{hdegcolor}\strut\ovalbox{\textup{{$\scriptscriptstyle#1$}}}}}
\newcommand{\hdeguphw}[2]{\overbrace{{\color{black}#2}\strut}^{\hidewidth\color{hdegcolor}\strut\ovalbox{\textup{{$\scriptscriptstyle#1$}}}\hidewidth}}
\newcommand{\hdegdownhw}[2]{\underbrace{{\color{black}#2}\strut}_{\hidewidth\color{hdegcolor}\strut\ovalbox{\textup{{$\scriptscriptstyle#1$}}}\hidewidth}}


\begin{abstract}
Let $\mathcal{L}$ be a finite-dimensional semisimple Lie algebra of
rank $N$ over an algebraically closed field of characteristic $0$.
Associated to $\mathcal{L}$ is a family of polynomial folding maps 
$$\textsf{F}_{n}:\mathbb{A}^N\to\mathbb{A}^N\quad\text{for}\quad n\ge1$$
having the property that $\textsf{F}_{n}$ has topological degree $n^N$ and
$$\textsf{F}_{m}\circ\textsf{F}_{n}=\textsf{F}_{n}\circ\textsf{F}_{m}\quad\text{for all}\quad m,n\ge1.$$
We derive formulas for the leading terms of the
folding maps on $\mathbb{A}^2$ associated to the Lie algebras
$\mathcal{A}_2$, $\mathcal{B}_2$, and $\mathcal{G}_2$, and we use these
formulas to compute the affine automorphism group of each folding
map.
\end{abstract}

\maketitle

\setcounter{tocdepth}{1}  

\tableofcontents

\section{Introduction}
\label{section:introduction}

The subject of commuting maps has been much studied in dynamics, see
for
example~\cite{MR288104,MR966870,MR946432,MR1509242,MR3515828,MR880608,MR1160332,MR1098340,MR937529,MR2314218}.
We start with a general definition, although this paper will
concentrate on the case of maps of the affine plane.

\begin{definition}
Let~$K$ be an algebraically closed field of characteristic~$0$,
let~$X/K$ be an algebraic variety, and let
\[
f : X \longrightarrow X
\]
be a finite morphism. We say that~$f$ is
\emph{permutable}\footnote{These maps have various names in the
literature, for example they are called \emph{ integrable}
in~\cite{MR909112,MR880608,MR1160332}.} if there exists a finite
morphism~$g:X\to{X}$ such that\textup:
\begin{itemize}
\setlength{\itemsep}{0pt}
\item
  $f$ and $g$ are not invertible, i.e., not automorphisms of~$X$.
\item
  $f\circ{g}=g\circ{f}$.
\item
  $f$ and $g$ do not have a common iterate.
\end{itemize}
\end{definition}

One problem that we will study in this paper is the computation of the
group of automorphisms, also sometime called the group of
self-similarities, of permutable maps.

\begin{definition}
Let~$K$ be an algebraically closed field of characteristic~$0$,
let~$X/K$ be a variety and let
\[
f : X \longrightarrow X
\]
be a finite morphism. The \emph{group of automorphisms of~$f$} is the
group
\[
\Aut(f) = \bigl\{ \a\in\Aut(X) : \a^{-1}\circ f\circ\a = f \bigr\}.
\]
\end{definition}

\begin{example}
A classical theorem of Julia~\cite{MR1509242} and
Ritt~\cite{MR1501252,MR1500495} (see also Eremenko~\cite{MR1027462})
classifies permutable maps on~$\AA^1$, and more generally on~$\PP^1$.
In particular, a permutable map~$f:\AA^1\to\AA^1$ is conjugate to
either a power map~$x^n$ or a Chebeyshev polynomial~$T_n$, whose
automorphism groups are given by
\[
\Aut(x^n) \cong \bfmu_{n-1}
\quad\text{and}\quad
\Aut(T_n) = \begin{cases}
  \bfmu_1 &\text{if $n$ is even,}\\
  \bfmu_2 &\text{if $n$ is odd.}\\
\end{cases}
\]
\end{example}

We note that one consturction of the Chebyshev polynomial~$T_n$ is
to start with the~$n$-power map on~$\GG_m$ and take the quotient by
the automorphism~$z\to{z^{-1}}$.  There is an analogous construction
in higher dimensions, where~$\GG_m$ is replaced by a
higher-dimensional torus associated to a Lie algebra and the quotient
is via an action of an affine Weyl group. The details of the
construction will not concern us, so we give only a brief sketch (at the end of this introduction)
and provide some references.

\begin{theorem}
\label{theorem:FoldnLieL}
\cite{MR909112,MR880608} Let~$\Lie{L}$ be a finite-dimensional
semisimple Lie algebra over an algebraically closed field of
characteristic~$0$, and let~$N$ be the rank of~$\Lie{L}$, i.e., the
dimension of a Cartan subalgebra. Then for each~$n\ge0$ there is an
associated polynomial map
\[
\Fold_n[\Lie{L}] : \AA^N \longrightarrow \AA^N,
\]
called the \emph{$n$th folding
map\footnote{For~$\Lie{B}_2$,~$\Lie{G}_2$, and~$\Lie{F}_4$, there are
additional folding
maps~$\Fold_{\!n\sqrt2}[\Lie{B}_2]$,~$\Fold_{\!n\sqrt3}[\Lie{G}_2]$,
and $\Fold_{\!n\sqrt2}[\Lie{F}_4]$ that satisfy the commutativity
property $\Fold_\a\circ\Fold_\b=\Fold_{\a\b}=\Fold_\b\circ\Fold_\a$.}
associated to the Lie algebra~$\Lie{L}$}, having the following
properties\textup:
\begin{parts}
\Part{(1)}
$\Fold_n[\Lie{L}]$ has topological degree~$n^N$.
\Part{(2)}
For all~$m,n\ge0$, we have
\begin{equation}
  \label{eqn:FmFnFmnFnFm}
  \Fold_m[\Lie{L}]\circ\Fold_n[\Lie{L}]
  = \Fold_{mn}[\Lie{L}]
  = \Fold_n[\Lie{L}]\circ\Fold_m[\Lie{L}].
\end{equation}
\end{parts}
\end{theorem}

\begin{definition}
The folding map~$\Fold_n[\Lie{L}]$ described in
Theorem~\ref{theorem:FoldnLieL} is so named because it may be
constructed geometrically in terms of folding the chambers
of~$\Lie{L}$ associated to a system of roots using the reflection maps
that generate the affine Weyl group of~$\Lie{L}$.  See~\cite{MR937529}
for a nice description of this construction in dimension~$2$.
\end{definition}

\begin{remark}
We observe that~$\Fold_n[\Lie{L}]$ for~$n\ge2$ is a permutable map
on~$\AA^N$, since consideration of topological degrees shows
that~$\Fold_m[\Lie{L}]$ and~$\Fold_n[\Lie{L}]$ cannot have a common
iterate if~$\gcd(m,n)=1$. More precisely, it follows
from~\eqref{eqn:FmFnFmnFnFm} that~$\Fold_m[\Lie{L}]$
and~$\Fold_n[\Lie{L}]$ have a common iterate if and only
if~\text{$m^i=n^j$} for some~\text{$i,j\ge1$}.  Further, it is known
that if two simple Lie algebras~$\Lie{L}_1$ and~$\Lie{L}_2$ have
non-isomorphic Weyl groups, then their folding maps are inequivalent;
see~\cite[Corollary to Proposition~3]{MR880608}.
\end{remark}

It is possible to construct higher dimensional permutable maps from
lower dimensional maps.  The following definition characterizes one
way that this can be done.

\begin{definition}
Let~$f:\AA^N\to\AA^N$ be a polynomial map. We say that~$f$
is~\emph{triangular}\footnote{These maps are called \emph{reducible}
in~\cite{MR880608}.} if there is an affine automorphism\footnote{An
\emph{affine automorphism of~$\AA^N$} is an
element~$\ell\in\GL_N\rtimes\GG_a^N$, where~$\ell=(M,\bfb)$ acts
on~$\AA^N$ via matrix multiplication and
translation,~\text{$(M,\bfb)\cdot\bfv=M\bfv+\bfb$}.}
\[
\ell:\AA^N\to\AA^N
\]
so that
\[
  \ell^{-1}\circ f\circ\ell(\bfx) = \bigl[
  \underbrace{\f_1(x_1)}_{\text{just $x_1$}},\,
  \underbrace{\f_2(x_1,x_2)}_{\text{just $x_1,x_2$}},\,
    \underbrace{\f_3(x_1,x_2,x_3)}_{\text{just $x_1,x_2,x_3$}},
      \ldots \bigr].
\]
\end{definition}

A bold conjecture of Veselov\footnote{Quoting from~\cite{MR880608}:
``The following construction~\cite{MR909112} gives a series of
integrable mappings for arbitrary~$n$ which, the author supposes,
exhausts all integrable polynomial mappings.''}  suggested that every
non-triangular permutable map on affine space comes from the Lie
algebra construction described in Theorem~\ref{theorem:FoldnLieL}.
This conjecture is not entirely correct, since for example there are
non-triangular permutable maps of~$\AA^2$ in Dinh's
classification~\cite{MR1824960} that do not come from the Lie algebra
construction. But Lie algebra folding maps do form 
very interesting families of permutable maps.

In particular, the classification of simple Lie algebras implies that
the folding maps of~$\AA^2$ are associated to one of the Lie
algebras~$\Lie{A}_2$,~$\Lie{B}_2$, or~$\Lie{G}_2$.\footnote{The
classification of finite-dimensional simple Lie algebras consists of
four infinite families and a handful of exceptional cases.  However,
in dimension~$2$ there are isomorphisms
\text{$\Lie{D}_2\cong\Lie{A}_1\times\Lie{A}_1$} and
\text{$\Lie{B}_2\cong\Lie{C}_2$}, so there are only three distinct
cases.}  Withers~\cite{MR937529} gives recursion formulas for these
$2$-dimensional folding maps. In this paper we use Withers' recursions
to find the leading terms of the folding map coordinate functions,
which in turn we use to compute the automorphism groups of
the~$2$-dimensional folding maps. We state the latter result here.
For the former, see Proposition~\ref{proposition:Aleadingterms}
for~$\Lie{A}_n$, Proposition~\ref{proposition:XnYnexpansionforBn}
for~$\Lie{B}_n$, and Proposition~\ref{proposition:XnYnGnformulas}
for~$\Lie{G}_n$.

\begin{theorem}
The affine automorphism groups of the~$2$-dimensional folding maps are
as follows\textup:
\begin{align*}
  \Aut\bigl(\Fold_n[\Lie{A}_2]\bigr)
  &\cong\begin{cases}
  \Scal_3 &\text{if $n\equiv1\pmodintext3$,} \\
  \bfmu_2 &\text{if $n\not\equiv1\pmodintext3$.} \\
  \end{cases} \\[3\jot]
  \Aut\bigl(\Fold_n[\Lie{B}_2]\bigr)
  &\cong\begin{cases}
  \bfmu_2 &\text{if $n\equiv1\pmodintext1$,} \\
  1 &\text{if $n\equiv0\pmodintext0$.} \\
  \end{cases} \\[3\jot]
  \Aut\bigl(\Fold_n[\Lie{G}_2]\bigr)
  &\cong 1.
\end{align*}
\end{theorem}
\begin{proof}
See Theorem~\ref{theorem:AutAn} for~$\Lie{A}_n$,
Theorem~\ref{theorem:AutBn} for~$\Lie{B}_n$, and
Theorem~\ref{theorem:GnAut} for~$\Lie{G}_n$.
\end{proof}

As a further application of our leading-term formulas, in
Theorem~\ref{theorem:morphfoldQ} we describe the maps on~$\PP^2$
induced by homogenizing the folding maps on~$\AA^2$.  In particular,
we prove that~$\Fold_n[\Lie{A}_2]$ and~$\Fold_n[\Lie{B}_2]$ extend to
endomorphisms on~$\PP^2$, while~$\Fold_n[\Lie{G}_2]$ extends to a
rational map of~$\PP^2$ having either one or two points of
indeterminacy, depending on the parity of~$n$.

\begin{remark}
There has been a considerable amount of work aimed at classifying
permutable maps of~$\PP^N$. We mention in particular the paper of
Dinh~\cite{MR1824960} that classifies permutable maps of~$\AA^2$ that
extend to morphisms of~$\PP^2$, the paper of Dinh and
Sibony~\cite{MR1931758} that extends this work to endomorphisms
of~$\PP^N$ whose degrees are multiplicatively independent, and the
paper of Kaufmann~\cite{MR3784253} that covers the case that the
degrees are multiplicatively dependent.
\par
In terms of the classification in~\cite{MR1824960}, the folding
map~$\Fold_n[\Lie{B}_2]$, which extends to an endomorphism of~$\PP^2$,
corresponds to a Case~(4) map with associated map on the line at
infinity being the homogenized Chebyshev polynomial~$T_n$.
Explicitly, the map~$\Fold_n[\Lie{B}_2]$ is characterized by the
functional equation
\[
\Fold_n[\Lie{B}_2](u+v,uv) = \bigl(T_n(u)+T_n(v),\,T_n(u)\cdot T_n(v)\bigr);
\]
cf.~\cite{MR946432}, remembering that~$\Lie{B}_2\cong\Lie{C}_2$.  On
the other hand, the folding map~$\Fold_n[\Lie{A}_2]$, which also
extends to an endomorphisms of~$\PP^2$, does not fit into any of the four
cases described in~\cite{MR1824960}, although presumably it is
covered by the more general treatment in~\cite{MR1931758}.  Finally,
the folding map~$\Fold_n[\Lie{G}_2]$ does not extend to an endomorphism
of~$\PP^2$, and thus it is not included in any of the current
classifications.
\end{remark}

\begin{proof}[Proof sketch of Theorem \ref{theorem:FoldnLieL}
 \textup(apr{\`e}s~\cite{MR880608}\textup)]
We take~$K=\CC$.
Let~$\Lie{L}_0\cong\CC^N$ be a Cartan subalgebra of~$\Lie{L}$, let~$\Lie{L}_0^*$
be its dual, let~$w_1,\ldots,w_N$ be a system of fundamental weights
in~$\Lie{L}_0^*$, and let~$\Lambda$ be the lattice in~$\Lie{L}_0$ that is
dual to~\text{$\Span_\ZZ(w_1,\ldots,w_N)\subset\Lie{L}_0^*$}.  Further
let~${\textsf{W}}$ be the Weyl group with its action on~$\Lie{L}_0^*$.
For each~$1\le{k}\le{N}$, define an exponential map
\[
\f_k : \Lie{L}_0/\Lambda \longrightarrow\CC,\quad
\f_k(\bfx) = \sum_{\s\in{\textsf{W}}} e^{2\pi i \s(w_k)(\bfx)},
\]
and fit these maps together to give a map
\[
\Phi_{\Lie{L}} : \Lie{L}_0/\Lambda \longrightarrow\CC^n,\quad
\Phi_{\Lie{L}} = (\f_1,\ldots,\f_N).
\]
A theorem of Chevalley says that the algebra of exponential invariants
for the action of~${\textsf{W}}$ is the polynomial ring generated
by~$\f_1,\ldots,\f_N$. Hence for every~$n\ge2$ there is a unique
polynomial map~$F_n$ characterized by
\[
\Phi_{\Lie{L}}(n\bfx) = F_n\bigl(\Phi_{\Lie{L}}(\bfx)\bigr).
\]
For futher details, see~\cite{MR909112,MR880608},
as well a~\cite{MR966870,MR946432,MR1160332,MR1098340}.
\end{proof}

\section{Folding Maps for the Lie Algebra $\Lie{A}_2$}
\label{section:foldingAn}

We consider the folding maps~$\Fold_n[\Lie{A}_2]$. To ease notation in
this section, we let
\[
z=x+iy\quad\text{and}\quad 
A_n(z) = A_n(x,y) = \Fold_n[\Lie{A}_2](x,y).
\]
Then~\cite{MR937529} says that the folding maps~$A_n$ are characterized by
the formulas
\begin{gather}
  \label{eqn:Anvalsandrecursion1}
  A_0(z)=3,\quad A_1(z)=z,\quad A_2(z)=z^2-2\zbar,\\
  \label{eqn:Anvalsandrecursion2}
  A_n(z) = zA_{n-1}(z)-\zbar A_{n-2}(z) + A_{n-3}(z).
\end{gather}

\begin{remark}
Although we will use the description of the~$A_n$ in terms of~$z$
and~$\zbar$, we note that~\eqref{eqn:Anvalsandrecursion1}
and~\eqref{eqn:Anvalsandrecursion2} may be rewritten in terms
of~$xy$-coordinates by letting
\[
(X_n,Y_n) = A_n(x,y).
\]
Then
\[
A_0(x,y) = (3,0),\quad
A_1(x,y) = (x,y),\quad
A_2(x,y) = (x^2-y^2-2x,\,2xy+2y),
\]
and
\begin{align*}
  X_{n} &= x (X_{n-1}-X_{n-2}) - y (Y_{n-1}+Y_{n-2}) + X_{n-3}, \\
  Y_{n} &= x (Y_{n-1}-Y_{n-2}) + y (X_{n-1}+X_{n-2}) + Y_{n-3}.
\end{align*}
\end{remark}

\begin{remark}
It is clear from the recursion for~$A_n$ that~$A_n\in\ZZ[z,\zbar]$, so
in particular the effect of complex conjugation is
\[
\overline{A_n(z)} = A_n(\zbar).
\]
Or, if one wants to be more formal and write~$A_n(z,\zbar)$, then
\[
\overline{A_n(z,\zbar)} = A_n(\zbar,z).
\]
\end{remark}

\begin{definition}
\label{definition:bigOnotation}
We write~$O(d)$ for a polynomial in~$K[z,\zbar]$ or in~$K[x,y]$ whose
total degree is at most~$d$, i.e., writing~$f=g+O(d)$ means that~$f-g$
has the form
\[
\sum_{i+j\le d} a_{ij}z^i\zbar^j,
\quad\text{or equivalently,}\quad
\sum_{i+j\le d} a_{ij}x^iy^j.
\]
\end{definition}

\begin{lemma}
\label{lemma:Acirczeta}
\begin{parts}  
\Part{(a)}
Let~$\z\in\CC$ satisfy~$\z^3=1$. Then
\begin{equation}
  \label{eqn:Acirczeta}
  A_n(\z z) = \z^n A_n(\z).
\end{equation}
\Part{(b)}
The polynomial~$A_n$ has the form
\[
A_n(z) = \sum_{k=0}^n \sum_{\substack{i+j=k\\ i-j\equiv n\pmodintext{3}\\}} a_{ij}z^i\,\zbar^{k-i}.
\]
\end{parts}
\end{lemma}

\begin{proof}
(a)\enspace
The assumption that~$\z$ is a cube root of unity implies in particular
that~$|\z|=1$, so~$\overline\z=\z^{-1}$.  We prove
Lemma~\ref{lemma:Acirczeta} by induction on~$n$. The
equality~\eqref{eqn:Acirczeta} is clearly true for~$A_0(z)=3$
and~$A_1(z)=z$.  For~$A_2(z)=z^2-2\zbar$, we compute
\[
A_2(\z z) = \z^2 z^2 - 2 \overline{\z z}
= \z^2 z^2 - 2 \z^{-1} \zbar = \z^2A_2(z).
\]
Now assume that~\eqref{eqn:Acirczeta} is true up to~$A_{n-1}$. Then
\begin{align*}
  A_n(\z z)
  &= \z zA_{n-1}(\z z)-\overline{\z z} A_{n-2}(\z z) + A_{n-3}(\z z) \\
  &= \z^n zA_{n-1}(z)-\overline{\z}\,\z^{n-2}\zbar A_{n-2}( z) + \z^{n-3} A_{n-3}( z) \\
  &\omit\hfill\text{by the induction hypothesis,}\\
  &= \z^n \bigl(  zA_{n-1}( z)-\overline{ z} A_{n-2}( z) + A_{n-3}( z) \bigr)\\
  &\omit\hfill\text{since $\overline\z=\z^{-1}=\z^2$ for $\z\in\bfmu_3$,} \\
  &= \z^n A_n(z).
\end{align*}
\par\noindent(b)\enspace
We initially write~$A_n(z)$ as
\[
A_n(z) = \sum_{k=0}^n \sum_{i+j=k} a_{ij}z^i\zbar^j
\]
and use~(a) to show that certain coefficients vanish.  Let~$\z$ be a
primitive cube root of unity. Then~(a) implies that
\[
\z^n \sum_{k=0}^n \sum_{i+j=k} a_{ij}z^i\zbar^j
=
\sum_{k=0}^n \sum_{i+j=k} a_{ij}(\z z)^i(\overline{\z z})^j
=
\sum_{k=0}^n \sum_{i+j=k} \z^{i-j} a_{ij}z^i\zbar^j.
\]
Hence
\[
a_{ij}\ne0 \quad\Longrightarrow\quad \z^n=\z^{i-j}
\quad\Longrightarrow\quad
i-j\equiv n\pmodintext3.
\qedhere
\]
\end{proof}

\begin{proposition}
\label{proposition:Aleadingterms}
For all~$n\ge2$, the polynomial~$A_n\in\ZZ[z,\barz]$ satisfies
\begin{equation}
\label{eqn:Anvalsandrecursion3}
A_n(z) = z^n - n z^{n-2} \zbar + \frac{n^2-3n}{2} z^{n-4} \zbar^2 + O(n-3).
\end{equation}
\end{proposition}
\begin{proof}
Using~\eqref{eqn:Anvalsandrecursion1} and~\eqref{eqn:Anvalsandrecursion2},
we compute
\begin{align*}
A_3(z) &= z^3 - 3 z \zbar + 3, \\
A_4(z) &= z^4 - 4 z^2 \zbar + 4 z + 2 \zbar^2, \\
A_5(z) &= z^5 - 5 z^3 \zbar+ 5 z^2 + 5 z \zbar^2 - 5 \zbar.
\end{align*}
Hence~\eqref{eqn:Anvalsandrecursion3} is true
for~$n=3,4,5$.\footnote{For~$n=3$, the
formula~\eqref{eqn:Anvalsandrecursion3} appears to have a term of the
form~$z^{-1}\zbar^2$, but the coefficient~$\frac{n^2-3n}{2}$ vanishes,
so that term does not appear in the formula.}  Assume
that~\eqref{eqn:Anvalsandrecursion3} is true up to~$n-1$ for
some~\text{$n\ge5$}, so our induction assumption says that
\begin{equation}
  \label{eqn:Anvalsandrecursion4}
A_j(z) = z^j - j z^{j-2} \zbar + \frac{j^2-3j}{2} z^{j-4} \zbar^2 + O(j-3)
\quad\text{for all $j<n$.}
\end{equation}
Then
\begin{align*}
 A_n
&=  zA_{n-1}-\zbar A_{n-2} + A_{n-3} 
\quad\text{from \eqref{eqn:Anvalsandrecursion2},} \\
&=  z\left( z^{n-1} - (n-1) z^{n-3} \zbar + \frac{n^2-5n+4}{2} z^{n-5} \zbar^2 + O(n-4) \right) \\
&\qquad{} - \zbar\left( z^{n-2} - (n-2) z^{n-4} \zbar + \frac{n^2-7n+10}{2} z^{n-6} \zbar^2 + O(n-5) \right)\\
&\qquad{} + \left( z^{n-3} - (n-3) z^{n-5} \zbar + \frac{n^2-9n+18}{2} z^{n-7} \zbar^2
+ O(n-6) \right) \\
&\omit\hfill using the induction hypothesis \eqref{eqn:Anvalsandrecursion4}, \\
&=  \left( z^{n} - (n-1) z^{n-2} \zbar + \frac{n^2-5n+4}{2} z^{n-4} \zbar^2 + O(n-3) \right) \\
&\qquad{} - \Bigl( z^{n-2}\zbar - (n-2) z^{n-4} \zbar^2 + O(n-3) \Bigr) + O(n-3) \\
&= z^n - nz^{n-2}\zbar + \frac{n^2-3n}{2}z^{n-4}\zbar^2+O(n-3).
\qedhere
\end{align*}
\end{proof}

We  use Proposition~\ref{proposition:Aleadingterms} to compute the
affine automorphism group of~$A_n$.

\begin{theorem}
\label{theorem:AutAn}
Let~$n\ge2$. Then 
\[
\Aut(A_n) \cong \begin{cases}
  \Scal_3  &\text{if $n\equiv1\pmodintext3$,} \\
  \bfmu_2  &\text{if $n\not\equiv1\pmodintext3$.} \\
\end{cases}
\]
Explicitly, if we let~$\z\in\bfmu_3$ be a primitive cube root of
unity, then the maps in~$\Aut(A_n)$ are
\[
\Aut(A_n) \cong \begin{cases}
  \{z,\,\z z,\,\z^2 z,\,\zbar,\,\z\zbar,\,\z^2\zbar\}
  &\text{if $n\equiv1\pmodintext3$,} \\
  \{z,\,\zbar\}
  &\text{if $n\not\equiv1\pmodintext3$.} \\
\end{cases}
\]
\end{theorem}

\begin{proof} 
Writing everything in terms of~$z=x+iy$ and~$\zbar=x-iy$, the elements
of~$\Aut(\AA^2)$ are identified with maps of the form
\[
\f_{\a,\b,\g}(z) = \a z + \b \zbar + \g
\quad\text{for}\quad \a,\b,\g\in\CC.
\]
(We note that~$\f_{\a,\b,\g}(z)$ is invertible if and only if~$|\a|^2\ne|\b|^2$.)
\par
Suppose that~$\f_{\a,\b,\g}\in\Aut(A_n)$ for some~$n\ge2$. Then
\begin{align*}
  \f_{\a,\b,\g} \circ A_n(z) &= A_n \circ \f_{\a,\b,\g}(z)
  \quad\text{by definition of $\Aut(A_n)$,} \\
  \a A_n(z) + \b \overline{A_n(z)} + \g
  &= A_n(\a z + \b \zbar + \g)
  \quad\text{by definition of $\f_{\a,\b,\g}$,}  \\
  \a \bigl(z^n + O(n-1)\bigr) & + \b \bigl(\zbar^n +O(n-1)\bigr) + \g \\
  &= (\a z + \b \zbar + \g)^n + O(n-1) \quad\text{Proposition~\ref{proposition:Aleadingterms},} \\
  \a z^n + \b \zbar^n
  &= (\a z+\b\zbar)^n + O(n-1)
\end{align*}
Looking at the coefficients of~$z^n$ and~$\zbar^n$ gives
\[
\a^n=\a\quad\text{and}\quad \b^n=\b,
\]
and then the lack of any other monomials on the
left-hand side forces the equality~\text{$\a\b=0$}. (This is where we use~$n\ge2$
and~$\characteristic(K)=0$.)
\par
We consider first the case that~$\a^n=\a\ne0$ and~$\b=0$,
so in particular we see that
\begin{equation}
  \label{eqn:an1eq1abareqaneg1}
  \a^{n-1}=1
  \quad\Longrightarrow\quad
  |\a|=1
  \quad\Longrightarrow\quad
  \overline\a = \a^{-1}.
\end{equation}
Using~$\f_{\a,0,\g}(z)=\a{z}+\g$, we compute
\begin{align*}
  \f_{\a,0,\g} \circ A_n(z) &= A_n \circ \f_{\a,0,\g}(z)
  \quad\text{by definition of $\Aut(A_n)$,} \\*
  \a A_n(z) + \g
  &= A_n(\a z + \g)
  \quad\text{by definition of $\f_{\a,0,\g}$,}  \\*
  \a \bigl(z^n-nz^{n-2}\zbar & + O(n-2)\bigr) + \g \\*
  &= (\a z + \g)^n - n (\a z + \g)^{n-2}(\overline\a\,\zbar + \overline\g) + O(n-2) \\*
  &\omit\hfill\text{from Proposition~\ref{proposition:Aleadingterms},} \\*[-2\jot]
  \a z^n  - \a n z^{n-2}\zbar &+ O(n-2) \\*
  &= \a^n z^n + n \a^{n-1}\g z^{n-1} - n \overline\a\,\a^{n-2} z^{n-2} \zbar + O(n-2).  
\end{align*}
This gives
\begin{equation}
  \label{eqn:nan1gzn10}
  n \a^{n-1}\g z^{n-1} = 0
  \quad\text{and}\quad
  -\a n z^{n-2}\zbar =  - n \overline\a\,\a^{n-2} z^{n-2} \zbar .
\end{equation}
The first equality tells us that~$\g=0$, and the second equality yields
\begin{align*}
\a  &=  \overline\a\cdot\a^{n-2} \quad\text{from \eqref{eqn:nan1gzn10},} \\
&= \a^{-1} \cdot \a^{-1}
\quad\text{from \eqref{eqn:an1eq1abareqaneg1} and using $\a^n=\a$.} 
\end{align*}
Hence~$\a^3=1$. This concludes the proof that
\begin{multline*}
\f_{\a,\b,\g}\in\Aut(A_n)\quad\text{and}\quad\a\ne0 \\
\;\Longrightarrow\;
\f_{\a,\b,\g}(z)=\f_{\a,0,0}(z)=\a z\quad\text{and}\quad\a^3=1.
\end{multline*}
\par
We leave to the reader the similar calculation for~$\a=0$ and~$\b^n=\b\ne0$,
which leads to the similar conclusion
\begin{multline*}
\f_{\a,\b,\g}\in\Aut(A_n)\quad\text{and}\quad\b\ne0 \\
\;\Longrightarrow\;
\f_{\a,\b,\g}(z)=\f_{0,\b,0}(z)=\b \zbar\quad\text{and}\quad\b^3=1.
\end{multline*}
\par
Let~$\z\in\bfmu_3$ be a cube root of unity.  It remains to check under
what circumstances the maps
\[
\f_{\z,0,0}(z)=\z z
\quad\text{and/or}\quad
\f_{0,\z,0}(z)=\z\,\zbar
\]
are in~$\Aut(A_n)$. 
Lemma~\ref{lemma:Acirczeta}(a) tells us that
\begin{equation}
  \label{eqn:Ancircphizetaz}
  A_n\circ\f_{\z,0,0}(z) = A_n(\z z) = \z^n A_n(z) = \z^{n-1}\cdot\f_{\z,0,0}\circ A_n(z).
\end{equation}
Hence for~$\z\in\bfmu_3$, we have
\[
\f_{\z,0,0}\in\Aut(A_n)
\quad\Longleftrightarrow\quad
\text{$\z=1$ or $n\equiv1\pmodintext3$.}
\]
We next observe that
\begin{align*}
A_n & \circ\f_{0,\z,0}(z) \\
&= A_n(\z\,\zbar) \quad\text{by definition of $\f_{0,\z,0}(z)=\z\,\zbar$,} \\
&= \z^n A_n(\zbar)
\quad\parbox[t]{.7\hsize}{since the formula in Lemma~\ref{lemma:Acirczeta}
is a formal identity in the variable~$z=x+iy$,
so it remains true if we replace~$z$ with~$\zbar=x-iy$, } \\
&= \z^{n-1} \cdot \z \cdot \overline{A_n(z)}
\quad\text{since~$A_n(z)\in\RR[z,\zbar]$, so $\overline{A_n(z)}=A_n(\zbar)$,} \\
&= \z^{n-1}\cdot\f_{0,\z,0}\circ A_n(z)
\quad\text{by definition of $\f_{0,\z,0}(z)=\z\,\zbar$.} 
\end{align*}
So for~$\z\in\bfmu_3$, we have
\[
\f_{0,\z,0}\in\Aut(A_n)
\quad\Longleftrightarrow\quad
\text{$\z=1$ or $n\equiv1\pmodintext3$.}
\]
This completes the proof that
\[
\Aut(A_n) =
\begin{cases}
  \{ \f_{\z,0,0} : \z\in\bfmu_3 \} \cup \{ \f_{0,\z,0} : \z\in\bfmu_3 \}
  &\text{if~\text{$n\equiv1\pmodintext3$},} \\
  \{ \f_{1,0,0},\,\f_{0,1,0} \}
  &\text{if~\text{$n\not\equiv1\pmodintext3$}.} \\
\end{cases}
\]
In in the case that~$n\equiv1\pmodintext3$, we may
identify~$\Aut(A_n)$ with~$\Scal_3$ by noting that the listed set of
maps is the group of $\RR$-linear transformations of~$\CC$ preserving
the equilateral triangle with vertices~$\{1,\z,\z^2\}$.
We also note that the map~$\f_{0,1,0}(z)=\zbar$ is simply the map~$(x,y)\to(x,-y)$,
so  in the case that~$n\not\equiv1\pmodintext3$, the group~$\Aut(A_n)$
is naturally identified with~$\bfmu_2$.
\end{proof}

\section{Folding Maps for the Lie Algebra $\Lie{B}_2$}
\label{section:foldingBn}
We consider the folding maps~$\Fold_n[\Lie{B}_2]$.
To ease notation in this section, we let
\begin{equation}
  \label{eqn:XnYnBninZxy}
  (X_n,Y_n) = B_n(x,y) = \Fold_n[\Lie{B}_2](x,y),
  \quad\text{so $X_n,Y_n\in\ZZ[x,y]$.}
\end{equation}
According to~\cite{MR937529}, the first few~$B_n$ are
\begin{align*}
  B_0(x, y) &= (4,4), \\
  B_1(x, y) &= (x, y),\\ 
  B_2(x, y) &= (x^2 -2y - 4, y^2 - 2x^2 + 4y + 4),\\
  B_3(x, y) &= (x^3 - 3xy - 3x, y^3- 3x^2y + 6y^2 + 9y),
\end{align*}
and then writing $B_n(x,y)=(X_n,Y_n)$, the subsequent~$B_n$ are given by
the recursion
\begin{align}
  \label{eqn:XnrecursionB2}
  X_{n+4} &= x(X_{n+3}+X_{n+1})-(2+y)X_{n+2}-X_n, \\
  \label{eqn:YnrecursionB2}
  Y_{n+4} &= y(Y_{n+3}+Y_{n+1})-(x^2-2y-2)Y_{n+2}-Y_n.
\end{align}

\begin{proposition}
\label{proposition:XnYnexpansionforBn}
With notation~$(X_n,Y_n)=B_n(x,y)$ as in~\eqref{eqn:XnYnBninZxy} and
with big-$O$ notation as described in~\eqref{definition:bigOnotation},
we have
\begin{align*}
  X_n & = x^n - nx^{n-2}y - nx^{n-2} + \frac{n(n-3)}{2} x^{n-4}y^2 + O(n-3) \\*
  Y_n &= y^n - nx^2y^{n-2} + x^3\cdot O(n-3) + O(n-1).
\end{align*}
\end{proposition}
\begin{proof}
The stated formulas are true for~$n\ge3$ from the explicit formulas for~$B_n$ in these
cases. So we assume that they are true up to~$n+3$ and prove that they
are true for~$n+4$ by induction.
\par
We use the recursion~\eqref{eqn:XnrecursionB2} for~$X_n$ and the
induction hypothesis to compute
\begin{align*}
  &X_{n+4} = xX_{n+3}+xX_{n+1}-(2+y)X_{n+2}-X_n \\
  &= x
  \left( x^{n+3} - (n+3)x^{n+1}y - (n+3)x^{n+1}   + \frac{n(n+3)}{2} x^{n-1}y^2 + O(n) \right) \\
  &\qquad + x\bigl( x^{n+1} + O(n) \bigr) \\
  &\qquad - (2+y) \Bigl( x^{n+2} - (n+2)x^{n}y + O(n) \Bigr)
  + O(n) \\
  &= \left( x^{n+4} - (n+3)x^{n+2}y  - (n+3)x^{n+1} + \frac{n(n+3)}{2} x^{n}y^2 + O(n+1) \right) \\
  &\qquad + x^{n+2} - 2 x^{n+2}  - x^{n+2}y + (n+2)x^{n}y^2 + O(n+1) \\
  &=  x^{n+4} - (n+4) x^{n+2}y - (n+4)x^{n+2} \\
  &\omit\hfill$\displaystyle{} + \frac{(n+4)(n+1)}{2} x^{n}y^2 + O(n+1)$. 
\end{align*}
\par
We next use the recursion~\eqref{eqn:YnrecursionB2} for~$Y_n$ and the induction
hypothesis to compute
\begin{align*}
  Y_{n+4}
  &= yY_{n+3} + yY_{n+1} - (x^2-2y-2)Y_{n+2} - Y_n \\
  &= y\Bigl(
  \hdegdown{\uparrow\uparrow}{y^{n+3} - (n+3)x^2y^{n+1}} + x^3\cdot O(n) + O(n+2) \Bigr) \\
  &\quad + y\Bigl( y^{n+1} - (n+1)x^2y^{n-1} + x^3\cdot O(n-2) + O(n) \Bigr)  \\
  &\quad - (\hdegdownhw{\uparrow\uparrow}{x^2}-2y-2)
  \Bigl( \hdegdown{\uparrow\uparrow}{y^{n+2}} - (n+2)x^2y^{n} + x^3\cdot O(n-1) + O(n+1) \Bigr) \\
  &\quad - \Bigl( y^n - nx^2y^{n-2} + x^3\cdot O(n-3) + O(n-1) \Bigr) \\
  &=  y^{n+4} - (n+4)x^2y^{n+2} + x^3\cdot O(n+1) + O(n+3),
\end{align*}
where in we have marked with~\ovalbox{$\scriptstyle\uparrow\uparrow$} the quantities that form
the terms that are not absorbed into the big-$O$ error
terms.  
\end{proof}

\begin{theorem}
\label{theorem:AutBn}
For all $n\ge2$, we have
\[
\Aut(B_n) = \begin{cases}
  \bfmu_2&\text{if $n$ is odd,} \\
  1&\text{if $n$ is even.}
  \end{cases}
\]
More precisely, when~$n$ is odd, the non-trivial
element of~$\Aut(B_n)$ is~$\f(x,y)=(-x,y)$.
\end{theorem}
\begin{proof}
We write
\[
(X_n,Y_n) = B_n(x,y) \quad\text{with $X_n,Y_n\in\ZZ[x,y]$,}
\]
and we suppose that
\[
\f(x,y) = (ax+by+c,dx+ey+f) \in \Aut(\AA^2)
\]
commutes with~$B_n$.
\par
For any monomial~$M\in\ZZ[x,y]$ and any polynomial~$P\in\ZZ[x,y]$,
we let
\[
\Coef(M,P) = \text{coefficient of the monomial $M(x,y)$ in $P(x,y)$,}
\]
and if~$V=(v_1,\ldots,v_N)$ is any list, we write
\[
V[i] = v_i = \text{$i$th entry of the list.}
\]
\par
We use the formulas for~$X_n$ and~$Y_n$ in
Proposition~\ref{proposition:XnYnexpansionforBn} to compute
\begin{align}
  \label{eqn:BNxyn1a}
  \Coef\Bigl( xy^{n-1},\, (\f\circ B_n)(x,y)[1] \Bigr)
  &= \Coef\Bigl( xy^{n-1},\, aX_n + bY_n + c \bigr) \notag\\
  &=  0,  \\
  \label{eqn:BNxyn1b}
  \Coef\Bigl( xy^{n-1},\, (B_n \circ\f) (x,y)[1] \Bigr) 
  &= B_n(ax+by+c, dx+ey+f)[1] \notag\\
  &= \Coef\bigl( xy^{n-1},\, (ax+by+c)^n \bigr) \notag\\
  &= nax(by)^{n-1}.
\end{align}
Equating~\eqref{eqn:BNxyn1a} and~\eqref{eqn:BNxyn1b}
yields~$ab=0$.
\par
We next look at the coefficient of~$x^2y^{n-2}$. 
\begin{align}
  \label{eqn:BNxyn1c}
  \Coef\Bigl( x^2y^{n-2},\, (\f\circ B_n)(x,y)[1] \Bigr)
  &= \Coef\Bigl( x^2y^{n-2},\, aX_n + bY_n + c \bigr) \notag\\
  &= -bn,  \\
  \label{eqn:BNxyn1d}
  \Coef\Bigl( x^2y^{n-2},\, (B_n \circ\f) (x,y)[1] \Bigr) 
  &= B_n(ax+by+c, dx+ey+f)[1] \notag\\
  &= \Coef\bigl( x^2y^{n-2},\, (ax+by+c)^n \bigr) \notag\\
  &= \binom{n}{2}(ax)^2(by)^{n-2}.
\end{align}
Equating~\eqref{eqn:BNxyn1c} and~\eqref{eqn:BNxyn1d} gives
\[
-bn = \frac{n(n-1)}{2} a^2b^{n-2},
\quad\text{which gives}\quad
b=0~\text{or}~\frac{n-1}{2} a^2b^{n-3} = -1.
\]
Since we know from earlier that~\text{$ab=0$}, and since the
invertibility of~$\f$ implies that~$a$ and~$b$ cannot both vanish,
this proves that~\text{$b=0$}.
\par
Using the fact that~$b=0$, we look at the first few terms in the expansion
of the~$x$-coordinates of~\text{$\f\circ{B_n}$} and~\text{$B_n\circ\f$}. Thus
\begin{align*}
  (\f & \circ B_n)(x,y)[1] \\*
  &=  aX_n + c \\*
  &= ax^n - anx^{n-2}y - anx^{n-2} + a\smash[t]{\frac{n(n-3)}{2}} x^{n-4}y^2 + O(n-3), \\
  (B_n & \circ\f) (x,y)[1] \\*
  &= B_n(ax+c, dx+ey+f)[1] \\*
  &= (ax+c)^n
  - n(ax+c)^{n-2}(dx+ey+f)
  - n(ax+c)^{n-2}\\*
  &\qquad 
  + \frac{n(n-3)}{2} (ax+c)^{n-4}(dx+ey+f)^2
  + O(n-3) \\
  &= \Bigl( a^n x^n + n a^{n-1} c x^{n-1} + \tbinom{n}{2} a^{n-2} c^2 x^{n-2} \Bigr) \\*
  &\qquad
  - \Bigl( n a^{n-2} d x^{n-1} + n a^{n-2} e x^{n-2} y + n a^{n-2} f x^{n-2} \\*
  &\qquad\qquad + n(n-2) a^{n-3} c d x^{n-2} + n(n-2) a^{n-3} c e x^{n-3} y \Bigr) \\*
  &\qquad
  - n a^{n-2} x^{n-2}
  + \Bigl( \tfrac{n(n-3)}{2} a^{n-4} d^2 x^{n-2}
  + \tfrac{n(n-3)}{2} a^{n-4} e^2 x^{n-4} y^2 \Bigr) \\*
  &\qquad + O(n-3) \\
  &=
  a^n x^n
  + \Bigl( n a^{n-1} c - n a^{n-2} d \Bigr) x^{n-1} - n a^{n-2} e x^{n-2} y \\
  &\qquad
  + \left( \tbinom{n}{2} a^{n-2} c^2 - n a^{n-2} f - n(n-2) a^{n-3} c d - n a^{n-2}
  + \tfrac{n(n-3)}{2} a^{n-4} d^2 \right) x^{n-2} \\*
  &\qquad 
    - n(n-2) a^{n-3} c e  x^{n-3} y 
    + \tfrac{n(n-3)}{2} a^{n-4} e^2 x^{n-4} y^2
    + O(n-3).
\end{align*}
Equating the~$x^n$ coefficients gives
\[
a=a^n,\quad\text{so}\quad a^{n-1}=1.
\]
Then equating the~$x^{n-1}$ coefficients gives
\[
na^{n-2} (ac-d) = 0,\quad\text{so}\quad d = ac.
\]
Then equating the~$x^{n-2}y$ coefficients and using~\text{$a^{n-1}=1$}
gives
\[
n a^{n-2} (a^2 - e ) = 0, \quad\text{so}\quad e = a^2.
\]
In particular, we have~$e\ne0$.  Then equating the~$x^{n-3}y$
coefficients and using \text{$a^{n-1}=1$} and~\text{$e\ne0$} gives
\[
n(n-2) a^{n-3} c e = 0,\quad\text{so}\quad c=0.
\]
We note that this also implies that~\text{$d=ac=0$}. To recapitulate, we have shown that
\[
a^{n-1}=1,\quad b=0,\quad c=0,\quad d=0,\quad e=a^2.
\]
Equating the~$x^{n-2}$ coefficients and using these values gives
\[
f = a^2 - 1.
\]
Hence
\begin{align}
  \label{eqn:phiaxeyfan11}
  \f(x,y) &= (ax,ey+f) \notag\\*
  &= (ax,a^2y+a^2-1) \quad
  \text{for some $a$ satisfying $a^{n-1}=1$.} 
\end{align}
\par
We next look at the~$x^2y^{n-2}$ terms of the~$y$-coordinate~$Y_n$ of~$B_n(x,y)$. Thus
\begin{align*}
  \Coef\Bigl( x^2y^{n-2},\,  (\f\circ B_n)&(x,y)[2] \Bigr) \\*
  &= \Coef\Bigl( x^2y^{n-2},\,  a^2 Y_n + f \bigr)
  \quad\text{from \eqref{eqn:phiaxeyfan11},} \\*
  &= -n a^2
  \quad\text{from Proposition~\ref{proposition:XnYnexpansionforBn},}  \\
  \Coef\Bigl( x^2y^{n-2},\,  (B_n \circ\f)& (x,y)[2] \Bigr)  \\*
  &= B_n(ax, a^2 y+f)[2]
  \quad\text{from \eqref{eqn:phiaxeyfan11},} \\*
  &= \Coef\bigl( x^2y^{n-2},\, -n(ax)^2(a^2y+f)^{n-2} \bigr) \\*
  &= -n a^2 (a^2)^{n-2}
  \quad\text{from Proposition~\ref{proposition:XnYnexpansionforBn},}  \\*
  &= -n \quad\text{since $a^{n-1}=1$.}
\end{align*}
Hence
\[
a^2 = 1,\quad\text{which implies that}\quad e = a^2 = 1 \quad\text{and}\quad f = a^2 -1 = 0.
\]
We have thus shown that
\[
\f(x,y)=(ax,y)\quad\text{with $a=\pm1$.}
\]
It remains to check when~$\f(x,y)=(-x,y)$ commutes with~$B_n$. More generally, we
will prove by induction that for all $n\ge0$, we have
\begin{equation}
  \label{eqn:XnYnBnpm1}
  X_n(-x,y)=(-1)^{n} X_n(x,y)
  \quad\text{and}\quad
  Y_n(-x,y)=Y_n(x,y).
\end{equation}
The explicit formulas for~$B_0,\ldots,B_3$ show that
that~\eqref{eqn:XnYnBnpm1} is true for~$0\le{n}\le3$.
Assume now that~\eqref{eqn:XnYnBnpm1} is true up to~$n+3$. Then
\begin{align*}
  X_{n+4}(-x,y) &= x\bigl(X_{n+3}(-x,y)+X_{n+1}(-x,y)\bigr) \\*
  &\qquad -(2+y)X_{n+2}(-x,y)-X_n(-x,y), \\*
  &= (-x)\bigl((-1)^{n+3}X_{n+3}(x,y)+(-1)^{n+1}X_{n+1}(x,y)\bigr) \\*
  &\qquad -(-1)^{n+2}(2+y)X_{n+2}(x,y) -(-1)^{n}X_n(x,y) \\*
  &= (-1)^{n+4}X_{n+4}(x,y), \\
  Y_{n+4}(-x,y) &= y\bigl(Y_{n+3}(-x,y)+Y_{n+1}(-x,y)\bigr) \\*
  &\qquad -\bigl((-x)^2-2y-2\bigr)Y_{n+2}(-x,y)-Y_n(-x,y) \\*
  &= y\bigl(Y_{n+3}(x,y)+Y_{n+1}(x,y)\bigr) \\*
  &\qquad -(x^2-2y-2)Y_{n+2}(x,y)-Y_n(x,y) \\*
  &= Y_{n+4}(x,y).
\end{align*}
It follows from~\eqref{eqn:XnYnBnpm1} that
\[
\f(x,y)=(-x,y) \in \Aut(B_n)
\quad\Longleftrightarrow\quad
n\equiv1\pmodintext{2}.
\qedhere
\]
\end{proof}

\section{Folding Maps for the Lie Algebra $\Lie{G}_2$}
\label{section:foldingGn}
We consider the folding maps~$\Fold_n[\Lie{G}_2]$.
To ease notation in this section, we let
\begin{equation}
  \label{eqn:XnYnGninZxy}
  (X_n,Y_n) = G_n(x,y) = \Fold_n[\Lie{G}_2](x,y),
  \quad\text{so $X_n,Y_n\in\ZZ[x,y]$.}
\end{equation}
The paper~\cite{MR937529} gives explicit formulas
for~$G_0,\ldots,G_5$, which for the convenience of the reader we have
reproduced in Figure~\ref{figure:Gn1to5hots}.  The subsequent values
of~$G_n(x,y)$ are then determined by the following recursions:
\begin{align}
  X_{n+6} &= x(X_{n+5}+X_{n+1})-(x+y+3)(X_{n+4}+X_{n+2}) \notag\\*
  &\qquad\qquad{} + (x^2-2y-4)X_{n+3}-X_n, \label{eqn:GXn6recursion}\\*
  Y_{n+6} &= y(Y_{n+5}+Y_{n+1})-(x^3-3xy-9x-5y-9))(Y_{n+4}+Y_{n+2}) \notag\\*
  &\qquad\qquad{} + (y^2-2x^3+6xy+18x+12y+8)Y_{n+3}-Y_n. \label{eqn:GYn6recursion}
\end{align}

\begin{figure}{
\small
\[
\begin{array}{|c|l|l|} \hline
  n & \text{$x$-coordinate of $G_n$} & \text{$y$-coordinate of $G_n$} \\ \hline\hline
    0 & 6 & 6 \\ \hline
    1
    & x & y \\ \hline
    2
    & x^2 - 2 x - 2 y - 6
    & \begin{array}{l}
      -2 x^3 + 6 x y + y^2 + 18 x + 10 y + 18
      \end{array}
    \\ \hline
    3 & x^3 - 3 x y - 9 x - 6 y - 12
    & \begin{array}{l}
       -3 x^3 y - 6 x^3 + 9 x y^2 + y^3 + 45 x y\\
       + 18 y^2 + 54 x + 63 y + 60\\
    \end{array}
    \\ \hline
    4 &
    \begin{array}{l}
       x^4 - 4 x^2 y - 10 x^2 - 4 x y \\
       + 2 y^2 - 8 x + 8 y + 6 \\
    \end{array}
    & \begin{array}{l}
          2 x^6 - 12 x^4 y - 4 x^3 y^2 - 36 x^4 - 28 x^3 y\\
         + 18 x^2 y^2 + 12 x y^3 + y^4 - 40 x^3 \\
         + 108 x^2 y + 120 x y^2 + 24 y^3 + 162 x^2 \\
         + 372 x y + 134 y^2 + 360 x + 280 y + 198
      \end{array}
    \\ \hline
    5 &
    \begin{array}{l}
       x^5 - 5 x^3 y - 15 x^3 - 5 x^2 y \\
       + 5 x y^2 - 10 x^2 + 35 x y \\
       + 10 y^2 + 55 x + 50 y + 60\\
    \end{array}
    & \begin{array}{l}
        5 x^6 y + 10 x^6 - 30 x^4 y^2 - 5 x^3 y^3 \\
        - 150 x^4 y - 65 x^3 y^2 + 45 x^2 y^3 + 15 x y^4 \\
        + y^5 - 180 x^4 - 205 x^3 y + 360 x^2 y^2\\
        + 240 x y^3 + 30 y^4 - 190 x^3 + 945 x^2 y\\
        + 1200 x y^2 + 255 y^3 + 810 x^2 + 2415 x y\\
        + 920 y^2 + 1710 x + 1495 y + 900\\
      \end{array}
    \\ \hline
\end{array}
\]
}
\label{figure:Gn1to5hots}
\caption{The first few $G_n$ polynomials}
\end{figure}

\begin{proposition}
\label{proposition:XnYnGnformulas}
We write~$(X_n,Y_n)=G_n(x,y)$ as in~\eqref{eqn:XnYnGninZxy} and
we use big-$O$ notation as described in~\eqref{definition:bigOnotation}.
\begin{parts}
\Part{(a)}
The values of~$X_0,\ldots,X_5$ are given in Figure~\ref{figure:Gn1to5hots}.
For~$n\ge5$, the top order terms of~$X_n$ satisfy
\begin{align}
\label{eqn:Xntop2terms}
X_n =
\smash[b]{
\hdegdownhw{\deg n}{x^n\vphantom{\frac{n^2}{2}}}
- \hdegdown{\deg n-1}{n x^{n-2}y\vphantom{\frac{n^2}{2}}}
+ \hdegdown{\deg n-2}{\frac{n^2-3n}{2} x^{n-4} y^2 - n x^{n-3}y - 3 n x^{n-2}}
}{\vrule height0pt width0pt depth3ex}
+ O(n-3) \\[-2ex]
\notag
\end{align}
\Part{(b)}
We have~$Y_0=6$, and for~$n\ge1$, the top order term of~$Y_n$ is given by
\[
  \label{eqn:Yntop2terms}
  Y_n = \begin{cases}
    (-1)^{n/2} 2 x^{3n/2} + O\left(\frac{3n-2}{2}\right)
    &\text{if $n$ is even,}\\[2\jot]
    (-1)^{(n-1)/2} n x^{(3n-3)/2} y + O\left(\frac{3n-3}{2}\right)
    &\text{if $n$ is odd.}\\
  \end{cases}
\]
\end{parts}
\end{proposition}
\begin{proof}
(a)\enspace
The formula~\eqref{eqn:Xntop2terms} for~$X_n$ is true
for~$n=5$ by inspection from Figure~\ref{figure:Gn1to5hots}.
Assume now that~\eqref{eqn:Xntop2terms} is true up to~$n+5$ for some~$n\ge0$.
The recursion~\eqref{eqn:GXn6recursion} says that
\begin{align}
  \label{eqn:GXnrecursion0}
  X_{n+6} &= x(X_{n+5}+X_{n+1}) - (x+y+3)(X_{n+4}+X_{n+2}) \notag\\*
  &\qquad\qquad{}
  + (x^2-2y-4)X_{n+3} - X_n
\end{align}
We use the induction hypothesis to compute the various pieces
of~$X_{n+6}$ up to~$O(n+3)$. 
\begin{align}
  \label{eqn:GXnrecursion1}
  x & (X_{n+5}+X_{n+1}) \notag\\*
  &=
  x^{n+6}-(n+5)x^{n+4}y
  +\frac{(n+5)(n+2)}{2} x^{n+2} y^2 - (n+5) x^{n+3}y \notag\\*
  &\qquad  - 3 (n+5) x^{n+4} + O(n+3)
  \quad\text{from \eqref{eqn:Xntop2terms}.} \\
  \label{eqn:GXnrecursion2}
  (x & +y+3)( X_{n+4}+X_{n+2})  \notag\\*
  &=   (x+y+3) X_{n+4} + O(n+3) \notag\\*
  &= (x+y+3)x^{n+4} - (x+y) (n+4) x^{n+2} y + O(n+3) \notag\\*
  &= x^{n+5} + x^{n+4} y + 3 x^{n+4} \notag\\*
  &\qquad - (n+4)x^{n+3} y - (n+4) x^{n+2} y^2 + O(n+3). \\
  \label{eqn:GXnrecursion3}
  (x^2 & -2y-4) X_{n+3}  -  X_n \notag\\*
  &= (x^2-2y)X_{n+3} + O(n+3) \notag\\*
  &= (x^2-2y)\bigl(x^{n+3} - (n+3)x^{n+1}y + O(n+1) \bigr) \notag\\*
  &= x^{n+5} - 2x^{n+3}y - (n+3)x^{n+3}y + O(n+3) \notag\\*
  &= x^{n+5} - (n+5)x^{n+3}y + O(n+3) .
\end{align}
Substituting~\eqref{eqn:GXnrecursion1},~\eqref{eqn:GXnrecursion2}
and~\eqref{eqn:GXnrecursion3} into~\eqref{eqn:GXnrecursion0} yields
\begin{align*}
  X_{n+6}
  &=  \biggl( x^{n+6}-(n+5)x^{n+4}y
  +\frac{(n+5)(n+2)}{2} x^{n+2} y^2 \\*
  &\omit\hfill\text{$\displaystyle{} - (n+5) x^{n+3}y  -  3 (n+5) x^{n+4}  \biggr)$} \\*
  &\qquad -  \Bigl( x^{n+5} + x^{n+4} y + 3 x^{n+4} 
                - (n+4)x^{n+3} y - (n+4) x^{n+2} y^2 \Bigr) \\*
  &\qquad +  \Bigl( x^{n+5} - (n+5)x^{n+3}y\Bigr) + O(n+3) \\*
  &= 
  x^{n+6} - (n+6) x^{n+4} y + \frac{(n+6)^2-3(n+6)}{2} x^{n+2} y^2\\*
  &\qquad - (n+6) x^{n+3} y - 3 (n+6) x^{n+4} + O(n+3).
\end{align*}
This completes the induction proof that~\eqref{eqn:Xntop2terms} holds
for all~$n\ge0$.
\par\noindent(b)\enspace
To ease notation,
we let
\begin{equation}
  \label{eqn:coeffYnG}
  \l_n = \begin{cases}
    (-1)^{n/2} 2
    &\text{if $n$ is even,}\\
    (-1)^{(n-1)/2} n
    &\text{if $n$ is odd.}\\
  \end{cases}
\end{equation}
An easy computation shows that the sequence~$\l_n$ described by~\eqref{eqn:coeffYnG}
satisfies the recursion
\begin{equation}
  \label{eqn:GYcoefrecursion}
  \l_{n+2} =
  \begin{cases}
    -\l_{n} &\text{if $n$ is even,} \\
    \l_{n+1}-\l_{n} &\text{if $n$ is odd.} \\
  \end{cases}
\end{equation}
\par
The formula~\eqref{eqn:Yntop2terms} for~$Y_n$ is true
for~$0\le{n}\le5$ by inspection from Figure~\ref{figure:Gn1to5hots}.
Assume now that~\eqref{eqn:Yntop2terms} is true up to~$n+5$.
\par
We consider first the case that~$n=2k$ is even.  Then the induction
hypothesis and the recursion~\eqref{eqn:GYn6recursion} give
\begin{align*}
  Y_{n+6}
  &=
  \hdegdown{ y\cdot O(3k+6) }{ y(Y_{2k+5}+Y_{2k+1}) }
  -
  \hdegdown{ (x^3+O(2))\cdot ( \l_{2k+4}x^{3k+6}+O(3k+5) ) }
           { (x^3-3xy-9x-5y-9)(Y_{2k+4}+Y_{2k+2}) }
  \\*
  &\qquad\qquad{} +
  \hdegdown{ O(3)\cdot O(3k+3) }{ (y^2-2x^3+6xy+18x+12y+8)Y_{2k+3} }
  -
  \hdegdown{ O(3k) }{ Y_{2k} },\\
  &=  -\l_{2k+4}x^{3k+9} + O(3k+8) \\
  &=  -\l_{n+4}x^{3(n+6)/2} + O\left(\frac{3(n+6)}{2}-1\right) \\
  &=  \l_{n+6}x^{3(n+6)/2} +  O\left(\frac{3(n+6)}{2}-1\right)
  \quad\text{from \eqref{eqn:GYcoefrecursion}.}
\end{align*}
\par
Similarly, if~$n=2k+1$ is even, then
\begin{align*}
  Y_{n+6}
  &=
  \hdegdown{ y\cdot (\l_{2k+6}x^{3k+9}+O(3k+8)) }{ y(Y_{2k+6}+Y_{2k+2}) }
  -
  \hdegdown{ (x^3+O(2))\cdot ( \l_{2k+5}x^{3k+6}y+O(3k+6) ) }
           { (x^3-3xy-9x-5y-9)(Y_{2k+5}+Y_{2k+3}) }
  \\*
  &\qquad\qquad{} +
  \hdegdown{ O(3)\cdot O(3k+6) }{ (y^2-2x^3+6xy+18x+12y+8)Y_{2k+4} }
  -
  \hdegdown{ O(3k+1) }{ Y_{2k+1} }, \\
  & = (\l_{2k+6}-\l_{2k+5})x^{3k+9}y + O(3k+9) \\
  &= (\l_{n+5}-\l_{n+4}) x^{(3n+15)/2} y + O\left(\frac{3n+15}{2}\right) \\
  &= \l_{n+6} x^{(3(n+6)-3)/2} y + O\left(\frac{3(n-6)-3}{2}\right)
  \quad\text{from \eqref{eqn:GYcoefrecursion}.} 
\end{align*}
This completes the induction proof that~\eqref{eqn:Yntop2terms} holds
for all~$n\ge1$.
\end{proof}

We use Proposition~\ref{proposition:XnYnGnformulas} to
compute the automorphism group of~$G_n$.

\begin{theorem}
\label{theorem:GnAut}
For all~$n\ge2$, we have
\[
\Aut(G_n) = 1.
\]
\end{theorem}
\begin{proof}
We write
\[
(X_n,Y_n) = G_n(x,y) \quad\text{with $X_n,Y_n\in\ZZ[x,y]$,}
\]
and we suppose that
\[
\f(x,y) = (ax+by+c,dx+ey+f) \in \Aut(\AA^2)
\]
commutes with~$G_n$.
\par
Our assumption that~$n\ge2$ implies that
\[
  \text{$n$ odd}
  \;\Longrightarrow\;
  \frac{3n-3}{2} \ge n
  \quad\text{and}\quad
  \text{$n$ even}
  \;\Longrightarrow\;
  \frac{3n-2}{2} \ge n.
\]
It follows from Proposition~\ref{proposition:XnYnGnformulas}(a)
that~$X_n=O(n)$, so the formulas for~$X_n$ and~$Y_n$ in
Proposition~\ref{proposition:XnYnGnformulas} give
\begin{align*}
  (\f\circ G_n)(x,y)[1]
  &= aX_n + bY_n + c \\
  &= 
  \begin{cases}
    bx^{3n/2} + O\left(\frac{3n-2}{2}\right)&\text{if $n$ is even,} \\
    bx^{(3n-3)/2}y + O\left(\frac{3n-3}{2}\right)&\text{if $n$ is odd,} \\
  \end{cases} \\*
 (G_n \circ\f) (x,y)[1] 
  &= G_n(ax+by+c, dx+ey+f)[1] \notag\\*
  &= (ax+by+c)^n + O(n-1)\\*
  &= O(n).
\end{align*}
We observe that if~$b\ne0$, then the highest degree term of the
composition $(\f\circ{G_n})(x,y)[1]$ is strictly larger than the
highest degree term of $(G_n\circ\f)(x,y)[1]$, so it would not be
possible for them to be equal.  Therefore
\[
b=0 \quad\text{and}\quad \f(x,y) = (ax+c,dx+ey+f).
\]
\par
We now use Proposition~\ref{proposition:XnYnGnformulas}(a)
to compare the next few terms in the expansions of the first
coordinates of~\text{$G_n\circ\f$} and~\text{$\f\circ{G_n}$}, keeping
in mind that~$b=0$. Thus
\begin{align}
  \label{eqn:Gncircfxy11}
 (G_n \circ\f)  (x,y)[1]
  &= G_n(ax+c, dx+ey+f)[1] \notag\\*
  &= (ax+c)^n
  - n(ax+c)^{n-2}(dx+ey+f) \notag\\*
  &\qquad + \frac{n^2-3n}{2}(ax+c)^{n-4}(dx+ey+f)^2 \notag\\*
  &\qquad  - n(ax+c)^{n-3}(dx+ey+f) \notag\\*
  &\qquad - 3n(ax+c)^{n-2}
  + O(n-3), \\*
  \label{eqn:fcircGnxy11}
  (\f \circ G_n)  (x,y)[1] 
  &= aX_n + c \notag\\*
  &= ax^n-anx^{n-2}y  + \smash[t]{ \frac{n^2-3n}{2} } ax^{n-4} y^2 \notag\\*
  &\qquad
  - n ax^{n-3}y
  - 3 n ax^{n-2}
  +O(n-3).
\end{align}
\par
We equate the degree~$n$ terms of~\eqref{eqn:Gncircfxy11}
and\eqref{eqn:fcircGnxy11}. This yields
\[
(ax)^n = \Deg_{n}\bigl((G_n \circ\f)(x,y)[1]\bigr)
= \Deg_{n}\bigl((\f \circ G_n)(x,y)[1]\bigr) = ax^n,
\]
and hence
\[
a^{n-1}=1.
\]
(We cannot have~$a=0$, since we already know that~$b=0$, and if~$a=b=0$,
the~$\f$ would not be invertible.)
\par
We next pull out the degree~$n-1$ terms of~\eqref{eqn:Gncircfxy11}
and\eqref{eqn:fcircGnxy11}, dividing by~$n$ to simplify the expressions.
Thus
\begin{align}
  \label{Degnneg1GnA}
  \frac{1}{n}\Deg_{n-1}\bigl((G_{n} \circ\f)(x,y)[1]\bigr)
  &= (ax)^{n-1}c - (ax)^{n-2}(dx+ey), \\*
  \label{Degnneg1GnB}
  \frac{1}{n}\Deg_{n-1}\bigl((\f\circ G_{n})(x,y)[1]\bigr)
  &= -ax^{n-2}y.
\end{align}
Setting~\eqref{Degnneg1GnA} equal to~\eqref{Degnneg1GnB},
multiplying by~$a$, and using~$a^{n-1}=1$,
we find that
\[
  (ac - d)x^{n-1} +(a^2 - e)x^{n-2}y = 0.
\]
Hence
\begin{equation}
  \label{eqn:deqaceeqa2}
  d = ac \quad\text{and}\quad e = a^2.
\end{equation}
\par
We next pull out the degree~$n-2$ terms of~\eqref{eqn:Gncircfxy11}
and\eqref{eqn:fcircGnxy11}, using~$a^n=a$ and dividing by~$n$ to
simplify the expressions. We find, after some algebra,  that
\begin{align*}
\frac{1}{n}\Deg_{n-2}\bigl((G_n & \circ\f)(x,y)[1]\bigr) \\*
&= 
\biggl(
\frac{n-1}{2} a^{-1}c^2
-  a^{-1} f - (n-2) a^{-2} c d \\*
&\hspace*{6em}
+ \frac{n-3}{2}a^{-3} d^2
- a^{-2} d
-3 a^{-1} 
\biggl) x^{n-2}  \\*
&+
\biggl(
2\frac{n-3}{2}a^{-3} de
- (n-2) a^{-2} c e
-a^{-2} e
\biggr) x^{n-3} y  \\*
&+  \biggl(\frac{n-3}{2}a^{-3}e^2\biggr)  x^{n-4}y^2.
\end{align*}
Similarly, the degree~$n-2$ terms of~$(\f \circ G_n)(x,y)[1]$ are
\begin{align*}
  \frac{1}{n}\Deg_{n-2}\bigl((\f \circ G_n)(x,y)[1]\bigr) 
  =
  - 3  ax^{n-2}
  -  ax^{n-3}y
  + \frac{n-3}{2} ax^{n-4} y^2.
\end{align*}
Setting
\[
\Deg_{n-2} \bigl( (G_n\circ\f)(x,y)[1]\bigr)
=
\Deg_{n-2} \bigl(\f\circ G_n)(x,y)[1]\bigr)
\]
we see that the coefficients of~$x^{n-4}y^2$ are equal, since we
already know from~\eqref{eqn:deqaceeqa2} that~\text{$e=a^2$}.
\par
Next equating the coefficients of~$x^{n-3}y$, we find that
\[
(n-3)a^{-3} de
- (n-2) a^{-2} c e
-a^{-2} e
=
-a.
\]
Using~$d=ac$ and~$e=a^2$ from~\eqref{eqn:deqaceeqa2}, after some algebra this becomes
\[
c = 1 - a.
\]
So we now know that
\begin{equation}
  \label{eqn:abcdevals}
    a^{n-1}=1,\quad
    b=0,\quad
    c=1-a,\quad
    d=ac=a-a^2,\quad
    e=a^2.
\end{equation}
\par
Finally, equating the coefficients of~$x^{n-2}$, we find that
\begin{align}
  \label{xn2coefsG}
    \frac{n-1}{2} a^{-1}c^2
    -  a^{-1} f - (n-2) a^{-2} c d 
    & + \frac{n-3}{2}a^{-3} d^2 \notag\\*
    & - a^{-2} d
    -3 a^{-1} + 3a = 0.
\end{align}
Substituting~\eqref{eqn:abcdevals} into~\eqref{xn2coefsG}
yields
\[
f = (a-1)(3a+4).
\]
Hence
\begin{equation}
  \label{eqn:phiax1aaa2xa2ya13a4}
  \f(x,y) = \bigl( ax+(1-a),\, (a-a^2)x + a^2y + (a-1)(3a+4) \bigr).
\end{equation}
In particular, we see that~$a=1$ implies that~$\f(x,y)=(x,y)$ is the identity map. 
\par
We now look at the~$y$-coordinates, where it is simpler to consider~$n$ even and odd
separately. We start with the case that~$n$ is even, say
\[
n=2k.
\]
Then
\begin{align*}
  (G_n \circ\f) & (x,y)[2]  \\
  &= G_n(ax+c, acx+a^2y+f)[2] \quad\text{from \eqref{eqn:abcdevals},} \\ 
  &= (-1)^{k} 2 (ax+c)^{3k} + O(3k-1)
  \quad\text{from Proposition~\ref{proposition:XnYnGnformulas}(b),} \\
  &= (-1)^{k} 2 a^{3k} x^{3k} + O(3k-1), \\
  (\f \circ G_n) & (x,y)[2] \\
  &= acX_n + a^2Y_n + f
  \quad\text{from \eqref{eqn:abcdevals},} \\
  &=  (-1)^{k} 2 a^2 x^{3k} + O(3k-1)
   \quad\text{from Proposition~\ref{proposition:XnYnGnformulas}(b).}
\end{align*}
Using~$G_n\circ\f=\f\circ{G_n}$, we find that
\[
(-1)^{k} 2 a^{3k} x^{3k} = (-1)^{k} 2 a^2 x^{3k} 
\quad\text{and hence that}\quad
a^{3k-2} = 1.
\]
But we also know from~\eqref{eqn:abcdevals} that~$a^{n-1}=a^{2k-1}=1$, and hence
\[
a^{\gcd(3k-2,2k-1)} = 1.
\]
We have
\[
2(3k-2) - 3(2k-1) = -1,\quad\text{so}\quad \gcd(3k-2,2k-1)=1.
\]
Hence if~$n$ is even, then~$a=1$.
\par
We next consider the case that~$x$ is odd, say
\[
n=2k+1.
\]
Then
\begin{align*}
  (G_n \circ\f)(x,y)[2] 
  &= G_n(ax+c, acx+a^2y+f)[2] \quad\text{from \eqref{eqn:abcdevals},} \\ 
  &=  (-1)^{k} n (ax+c)^{3k} (acx+a^2y+f) + O(3k) \\
  &\omit\hfill\text{from Proposition~\ref{proposition:XnYnGnformulas}(b),} \\
  &= (-1)^{k} n a^{3k}x^{k} (acx+a^2y) + O(3k), \\
  (\f \circ G_n) (x,y)[2] 
  &= acX_n + a^2Y_n + f
  \quad\text{from \eqref{eqn:abcdevals},} \\
  &= (-1)^{k} n a^2 x^{3k} y + O(3k)
   \quad\text{from Proposition~\ref{proposition:XnYnGnformulas}(b).}
\end{align*}
Using~$G_n\circ\f=\f\circ{G_n}$, we find that
\[
(-1)^{k} n a^{3k+1}cx^{3k+1} + (-1)^{k} n a^{3k+2}x^{3k}y
= (-1)^{k} n a^2 x^{3k} y.
\]
The coefficient of~$x^{3k+1}$ must vanish, which shows that~$c=0$, and
then~\eqref{eqn:abcdevals} forces~$a=1$. So we have proven that
if~$n$ is odd, then~\text{$a=1$}.
\par
To recapitulate, we have shown that~$a=1$ for all~$n\ge2$.
Then~\eqref{eqn:phiax1aaa2xa2ya13a4} tells us that~$\f(x,y)=(x,y)$ is
the identity map, which completes the proof that
\[
\Aut(G_n) = 1\quad\text{for all $n\ge2$.}
\qedhere
\]
\end{proof}

\section{Geometric Properties of $2$-Dimensional Folding Maps}
\label{section:geomprops}

In this section we study the maps on~$\PP^2$ induced
by~$2$-dimensional folding maps.

\begin{theorem}
\label{theorem:morphfoldQ}
\leavevmode
\begin{parts}
\Part{(a)}
Let~$\overline{A}_n(X,Y,Z)$ be the homogenization of~$A_n=\Fold_n[\Lie{A}_2]$.
Then
\[
\overline{A}_n:\PP^2\longrightarrow\PP^2\quad\text{is a morphism of degree $n$.}
\]
\Part{(b)}
Let~$\overline{B}_n(X,Y,Z)$ be the homogenization of~$B_n=\Fold_n[\Lie{B}_2]$.
Then
\[
\overline{B}_n:\PP^2\longrightarrow\PP^2\quad\text{is a morphism of degree $n$.}
\]
\Part{(c)}
Let~$\overline{G}_n(X,Y,Z)$ be the homogenization of~$G_n=\Fold_n[\Lie{G}_2]$. Then
\[
\overline{G}_n : \PP^2 \dashrightarrow \PP^2
\quad\text{is a rational map of degree }\left\lfloor\frac{3n}{2}\right\rfloor.
\]
If~$n\ge2$, then~$\overline{G}_n$ is not a morphism, and its
indeterminacy locus is given by
\begin{equation}
  \label{eqn:IGn}
  I(\overline{G}_n) = \begin{cases}
    \bigl\{[0,1,0]\bigr\}&\text{if $n$ is even,} \\
    \bigl\{[0,1,0],[1,0,0]\bigr\}&\text{if $n$ is odd.} \\
  \end{cases}
\end{equation}
The dynamical degree of~$\overline{G}_n$ is~$\operatorname{DynDeg}(\overline{G}_n)=n$.
\end{parts}
\end{theorem}

\begin{proof}
(a)\enspace
Proposition~\ref{proposition:Aleadingterms} tells us that
\[
A_n(x,y) = \Bigl( \Real (x+iy)^n + O(n-1),\, \Imag (x+iy)^n + O(n-1) \Bigr),
\]
so when we homogenize, we get
\[
\overline{A}_n(X,Y,Z) =
\Bigl[
  \Real(X+iY)^n + Z\hdegdownhw{\text{homog.\ deg $n-1$}}{u(X,Y,Z)},\,
  \Imag(X+iY)^n + Z\hdegdownhw{\text{homog.\ deg $n-1$}}{v(X,Y,Z)},\,  Z^n \Bigr].
\]
Then
\begin{align*}
  \overline{A}_n(X,Y,Z) = [0,0,0]
  &\quad\Longleftrightarrow\quad
  \Real(X+iY)^n = \Imag(X+iY)^n = Z = 0 \\
  &\quad\Longleftrightarrow\quad
  (X+iY)^n = Z = 0 \\
  &\quad\Longleftrightarrow\quad
  X = Y = Z = 0.
\end{align*}
Hence~$\overline{A}_n$ is a morphism of degree~$n$.  
\par\noindent(b)\enspace
Proposition~\ref{proposition:XnYnexpansionforBn} tells us that
\[
B_n(x,y) = \Bigl( x^n + O(n-1),\,
y^n + x^2\cdot O(n-2) + O(n-1) \Bigr),
\]
so when we homogenize, we get
\[
\overline{B}_n(X,Y,Z) =
\Bigl[
  X^n + Z\hdegdownhw{\text{homog.\ deg $n-1$}}{u(X,Y,Z)},\,
  Y^n + X^2\hdegdownhw{\text{homog.\ deg $n-2$}}{v(X,Y,Z)}
  + Z\hdegdownhw{\text{homog.\ deg $n-1$}}{w(X,Y,Z)},\,
  Z^n
  \Bigr].
\]
Setting~$\overline{B}_n(X,Y,Z)=[0,0,0]$, the third coordinate
gives~$Z=0$, then the first coordinate gives~$X=0$, and then the
second coordinate gives~$Y=0$.
Hence~$\overline{B}_n$ is a morphism of degree~$n$.  
\par\noindent(c)\enspace  
Suppose first that~$n$ is even, say~$n=2k$. Then
Proposition~\ref{proposition:XnYnGnformulas} tells us that
\[
G_n(x,y) = (x^{2k}+\cdots,(-1)^k2x^{3k}+\cdots).
= \Bigl( x^{2k} + O(2k-1),\, (-1)^k2x^{3k} + O(3k-1) \Bigr).
\]
Hence
\[
\overline{G}_n(X,Y,Z) = \bigl[ X^{2k}Z^k + Z\hdegdownhw{\text{homog.\ deg $3k-1$}}{u(X,Y,Z)},
(-1)^k2X^{3k} + Z\hdegdownhw{\text{homog.\ deg $3k-1$}}{v(X,Y,Z)},
Z^{3k} \bigr].
\]
From this we can read off
\[
\deg\overline{G} = 3k =\frac{3n}{2}
\quad\text{and}\quad
I(\overline{G}_n)
= \{X=Z=0\}.
\]
\par
Next suppose that~$n=2k+1$ is odd for some~$k\ge1$. Again using
Proposition~\ref{proposition:XnYnGnformulas}, we find that
\[
G_n(x,y) = \Bigl( x^{2k+1}+O(2k),(-1)^k(2k+1)x^{3k}y+O(3k-1)\Bigr).
\]
Hence
\begin{multline*}
  \overline{G}_n(X,Y,Z) = \bigl[ X^{2k+1}Z^{k} +
    \smash[t]{ Z\hdeguphw{\text{homog.\ deg $3k$}}{u(X,Y,Z)}, } \\*
(-1)^k(2k+1)X^{3k}Y + Z\hdegdownhw{\text{homog.\ deg $3k$}}{v(X,Y,Z)},
Z^{3k+1} \bigr].
\end{multline*}
From this and the assumption that~$k\ge1$ we can read off
\begin{align*}
  \deg\overline{G} &= 3k+1 = \frac{3n-1}{2} = \left\lfloor\frac{3n}{2}\right\rfloor
  \quad\text{since $n=2k+1$ is odd,} \\
I(\overline{G}_n)
&= \{X=Z=0\} \cup \{Y=Z=0\}.
\end{align*}
\par
This completes the proof of~(c) except for the computation of the dynamical degree.
Using~$\deg\overline{G}_n=\lfloor3n/2\rfloor$, 
the limit formula definition of the dynamical degree yields
\begin{align*}
\operatorname{DynDeg}(\overline{G}_n)
&= \lim_{m\to\infty} (\deg \overline{G}_n^m)^{1/m} \\*
&= \lim_{m\to\infty} (\deg \overline{G}_{n^m})^{1/m} \\*
&= \lim_{m\to\infty} \left\lfloor \frac{3n^m}{2}\right\rfloor^{1/m}
= n.
\qedhere
\end{align*}
\end{proof}

\begin{remark}
As noted in the footnote to Theorem~\ref{theorem:FoldnLieL}, the Lie
group~$\Lie{B}_2$ admits some additional folding maps arising from the
fact that a right isosceles triangle can be folded in half to form two
right isosceles triangles.  More precisely, there is a~$\Lie{B}_2$
folding map
\[
B_{\sqrt2}(x,y)= (y,x^2-2y-4)
\quad\text{satisfying}\quad
B_{\sqrt2}^2=B_2,
\]
leading to additional folding maps via
$B_{n\sqrt2}=B_{n}\circ{B_{\sqrt2}}$. The associated homogenized map
on~$\PP^2$ is
\[
  \overline{B}_{\sqrt2}(X,Y,Z)= [YZ,X^2-2YZ-4Z^2,Z^2],
\]
which is not a morphism, since it is not defined at $[0,1,0]$.
So~$\overline{B}_{\sqrt2}$ is an example of a non-morphism whose
second iterate~$\overline{B}_2$ is a morphism.
\end{remark}

\begin{remark}
As noted in the footnote to Theorem~\ref{theorem:FoldnLieL}, the Lie
group~$\Lie{G}_2$ admits some additional folding maps arising from the
fact that a~\text{30-60-90} triangle can be folded onto
itself~\text{3-to-1}.
More precisely, there is a~$\Lie{G}_2$
folding map
\[
G_{\sqrt3}(x,y)=(y,x^3-3xy-9x-6y-12)
\quad\text{satisfying $G_{\sqrt3}^2=G_3$,}
\]
leading to additional folding maps via
$G_{n\sqrt3}=G_{n}\circ{G_{\sqrt3}}$.
The associated homogenized map on~$\PP^2$ is
\[
\overline{G}_{\sqrt3}(X,Y,Z)= [YZ^2,X^3-3XYZ-9XZ^2-6YZ^2-12Z^3,Z^3].
\]
It is a rational map with indeterminacy locus~$[0,1,0]$. 
\end{remark}


\begin{acknowledgement}
The author would like to thank Max Weinreich and Tien-Cuong Dinh for
some helpful remarks.  Various numerical and algebraic calculations
were done using Pari-GP~\cite{PariGP} and Magma~\cite{MR1484478}.  The
author's research was partially supported by Simons Collaboration
Grant \#712332.
\end{acknowledgement}



\begin{thebibliography}{10}

\bibitem{MR288104}
E.~A. Bertram.
\newblock Polynomials which commute with a {T}chebycheff polynomial.
\newblock {\em Amer. Math. Monthly}, 78:650--653, 1971.

\bibitem{MR1484478}
Wieb Bosma, John Cannon, and Catherine Playoust.
\newblock The {M}agma algebra system. {I}. {T}he user language.
\newblock {\em J. Symbolic Comput.}, 24(3-4):235--265, 1997.
\newblock Computational algebra and number theory (London, 1993).

\bibitem{MR966870}
O.~A. Chalykh.
\newblock Some properties of polynomial mappings that are connected with {L}ie
  algebras.
\newblock {\em Vestnik Moskov. Univ. Ser. I Mat. Mekh.}, (3):57--59, 1988.

\bibitem{MR1824960}
Tien-Cuong Dinh.
\newblock Sur les endomorphismes polynomiaux permutables de {${\mathbb C}^2$}.
\newblock {\em Ann. Inst. Fourier (Grenoble)}, 51(2):431--459, 2001.

\bibitem{MR1931758}
Tien-Cuong Dinh and Nessim Sibony.
\newblock Sur les endomorphismes holomorphes permutables de {${\mathbb P}^k$}.
\newblock {\em Math. Ann.}, 324(1):33--70, 2002.

\bibitem{MR1027462}
A.~\`E. Er\"{e}menko.
\newblock Some functional equations connected with the iteration of rational
  functions.
\newblock {\em Algebra i Analiz}, 1(4):102--116, 1989.

\bibitem{MR946432}
Michael~E. Hoffman and William~Douglas Withers.
\newblock Generalized {C}hebyshev polynomials associated with affine {W}eyl
  groups.
\newblock {\em Trans. Amer. Math. Soc.}, 308(1):91--104, 1988.

\bibitem{MR1509242}
Gaston Julia.
\newblock M\'emoire sur la permutabilit\'e des fractions rationnelles.
\newblock {\em Ann. Sci. \'Ecole Norm. Sup. (3)}, 39:131--215, 1922.

\bibitem{MR3784253}
Lucas Kaufmann.
\newblock Commuting pairs of endomorphisms of {$\mathbb P^2$}.
\newblock {\em Ergodic Theory Dynam. Systems}, 38(3):1025--1047, 2018.

\bibitem{MR3515828}
\"{O}mer K\"{u}\c{c}\"{u}ksakall\i.
\newblock Bivariate polynomial mappings associated with simple complex {L}ie
  algebras.
\newblock {\em J. Number Theory}, 168:433--451, 2016.

\bibitem{MR1501252}
J.~F. Ritt.
\newblock Permutable rational functions.
\newblock {\em Trans. Amer. Math. Soc.}, 25(3):399--448, 1923.

\bibitem{MR1500495}
J.~F. Ritt.
\newblock Errata: ``{P}ermutable rational functions'' [{T}rans. {A}mer. {M}ath.
  {S}oc. {\bf 25} (1923), no. 3, 399--448; 1501252].
\newblock {\em Trans. Amer. Math. Soc.}, 26(4):494, 1924.

\bibitem{PariGP}
{The PARI~Group}, Univ. Bordeaux.
\newblock {\em {PARI/GP version {\tt 2.13.3}}}, 2021.
\newblock available from \url{http://pari.math.u-bordeaux.fr/}.

\bibitem{MR909112}
A.~P. Veselov.
\newblock Integrable polynomial mappings and {L}ie algebras.
\newblock In {\em Geometry, differential equations and mechanics ({R}ussian)
  ({M}oscow, 1985)}, pages 59--63. Moskov. Gos. Univ., Mekh.-Mat. Fak., Moscow,
  1986.

\bibitem{MR880608}
A.~P. Veselov.
\newblock Integrable mappings and {L}ie algebras.
\newblock {\em Dokl. Akad. Nauk SSSR}, 292(6):1289--1291, 1987.
\newblock English translation: Soviet Math. Dokl. 35 (1987), no. 1, 211--213.

\bibitem{MR1160332}
A.~P. Veselov.
\newblock Integrable mappings.
\newblock {\em Uspekhi Mat. Nauk}, 46(5(281)):3--45, 190, 1991.
\newblock English translation: Russian Math. Surveys 46 (1991), no. 5, 1--51.

\bibitem{MR1098340}
A.~P. Veselov.
\newblock What is an integrable mapping?
\newblock In {\em What is integrability?}, Springer Ser. Nonlinear Dynam.,
  pages 251--272. Springer, Berlin, 1991.

\bibitem{MR937529}
Wm.~Douglas Withers.
\newblock Folding polynomials and their dynamics.
\newblock {\em Amer. Math. Monthly}, 95(5):399--413, 1988.

\bibitem{MR2314218}
Karl Zimmermann.
\newblock Commuting polynomials and self-similarity.
\newblock {\em New York J. Math.}, 13:89--96, 2007.

\end{thebibliography}

\end{document}